\documentclass[12pt,reqno]{amsart}

\usepackage{amsthm, mathrsfs,amssymb,amsmath}
\usepackage{enumerate}
\usepackage[hidelinks]{hyperref}
\usepackage{xcolor}
\hypersetup{
	colorlinks,
	linkcolor={red!50!black},
	citecolor={green!50!black},
	urlcolor={red!80!black}
}
\makeatletter
\@namedef{subjclassname@2010}{%
	\textup{2010} Mathematics Subject Classification}
\makeatother

\allowdisplaybreaks

\newtheorem{thm}{Theorem}[section]
\newtheorem{lem}[thm]{Lemma}

\newtheorem{cor}[thm]{Corollary}

\theoremstyle{definition}
\newtheorem{defn}[thm]{Definition} 

\theoremstyle{remark}

\title[Titchmarsh theorems on Damek-Ricci spaces]{Titchmarsh theorems on Damek-Ricci spaces via moduli of continuity of higher order}

\author{Manoj Kumar}
\address{Manoj Kumar \endgraf
	Department of Mathematics
	\endgraf
	Indian Institute of Science Bangalore, India}
\endgraf
\email{manojk9t3@gmail.com}

\author{Vishvesh Kumar}
\address{Vishvesh Kumar \endgraf
	Department of Mathematics: Analysis, Logic and Discrete Mathematics
	\endgraf
	Ghent University, Belgium}
\endgraf
\email{vishveshmishra@gmail.com, Vishvesh.Kumar@UGent.be}

\author[Michael Ruzhansky]{Michael Ruzhansky}
\address{
	Michael Ruzhansky:
	\endgraf
	Department of Mathematics: Analysis, Logic and Discrete Mathematics
	\endgraf
	Ghent University, Belgium
	\endgraf
	and
	\endgraf
	School of Mathematical Sciences
	\endgraf
	Queen Mary University of London
	\endgraf
	United Kingdom
	\endgraf
	{\it E-mail-} {\rm michael.ruzhansky@ugent.be}
}

\allowdisplaybreaks

\begin{document}
	
	\begin{abstract} 
		
		A classical theorem of Titchmarsh relates the $L^2$-Lipschitz functions and decay of the Fourier transform of the functions. In this note, we prove the Titchmarsh theorem for Damek-Ricci space (also known as harmonic $NA$ groups) via moduli of continuity of higher orders. We also prove an analogue of another Titchmarsh theorem which provides integrability properties of the Fourier transform for functions in the H\"older Lipschitz spaces.
		
	\end{abstract}
	\maketitle
	
	\section{Introduction}  \label{Sec1}
	
	The classical Titchmarsh theorem \cite[Theorem 85]{Titch} characterizes the $L^2$-Lipschitz functions in terms of certain decay of the Fourier transform of the functions. It can be stated as follows: Let $\alpha\in (0,1)$ and $f\in L^2(\mathbb{R}).$ Then, $\|\tau_tf-f\|_2\leq C_1t^\alpha$ for all sufficiently small $t>0,$ if and only if 
	$$\int_{|\xi|>\frac{1}{t}}|\widehat{f}(\xi)|^2\,d\xi\leq C_2t^{2\alpha},$$ for all sufficiently small $t>0.$ Here, $\widehat{f}$ is the Fourier transform of $f$ and $\tau_t$ is the translation operator. 
	
	The Titchmarsh theorem has been extensively studied in many different contexts on various groups, for instances, the higher dimensional Euclidean spaces \cite{Younis,Bray}, the Vilenkin groups \cite{Younis2}, the special linear group of real matrices of order two $SL_2(\mathbb{R})$ \cite{Younis1}, the rank one symmetric spaces of non-compact type \cite{Platonov,FBK}, the $p$-adic groups \cite{Platonov17} and the compact homogeneous manifolds \cite{DDR}. In terms of the moduli of continuity, the theorem has been explored on $\mathbb{R}$ \cite{Volo,DFR} and the rank one symmetric spaces \cite{FRS}. See \cite{BrayPinsky,GT,J} for some growth properties of the Fourier transform on certain spaces via moduli of continuity.
	
Damek-Ricci spaces, also known as Harmonic $NA$ groups, are natural generalizations of the Iwasawa $NA$ groups of the real rank one simple Lie groups. Particularly, the rank one symmetric spaces of non-compact type form a subclass of the Damek-Ricci spaces. In general, the Damek-Ricci spaces need not be symmetric spaces. 
	
Very recently, Titchmarsh type results for Damek-Ricci spaces were explored in  \cite{KR,ED}. In this note, we extend the classical Titchmarsh theorem to the setting of Damek-Ricci spaces in terms of the moduli of continuity of higher orders (see Corollary \ref{TitThm}). We note here that it is new on the Damek-Ricci spaces even in the case of moduli of continuity of order one.

\cite[Theorem 1.5 (A)]{J} provides certain decay properties of the Helgason Fourier transform for functions in the generalized Besov spaces. We also prove an analogue of this theorem on Damek-Ricci spaces. 

Another Titchmarsh theorem \cite[Theorem 84]{Titch} provides the integrability properties of the Fourier transform of functions belonging to the H\"older-Lipschitz spaces. This theorem has been studied for various groups, for examples, $SL_2(\mathbb{R})$ \cite{Younis1}, the Euclidean space \cite{Bray}, the compact homogeneous manifolds \cite{DDR} and the Damek-Ricci spaces \cite{KR,ED}. We also prove a generalization of this Titchmarsh theorem over Damek-Ricci spaces.

\section{Preliminaries}
In this section, we recall basics required about the Damek-Ricci spaces and moduli of continuity. Throughout the paper, we denote by $C,C_1,C_2,...$ constants whose values may vary from one line to the other.
	\subsection{Fourier analysis on Damek-Ricci spaces} \label{Ess}
	
	For the details about the analysis and geometry of Damek-Ricci spaces and associated Fourier analysis, one can refer to \cite{Damek,DamekRicci,DamekRicci1, Di-Blasio,Anker96,ACB97,CDK91, RS, KRS, Kaplan, Astengo, Astengo2, KR1}. 
	
	Let $\mathfrak{n}$ be a two-step real nilpotent Lie algebra, equipped with an inner product $\langle \,,\, \rangle.$ Let $\mathfrak{z}$ denote the center of $\mathfrak{n}$ and let $\mathfrak{v}$ denote the orthogonal complement of $\mathfrak{z}$ in $\mathfrak{n}$ w.r.t. $\langle \,,\, \rangle.$ 
	Suppose that the dimensions of $\mathfrak{v}$ and $\mathfrak{z}$ are denoted by $m$ and $k$ respectively as real vector spaces. The Lie algebra $\mathfrak{n}$ is said to be $H$-type algebra if for each $Z \in \mathfrak{z},$ the map $J_Z:\mathfrak{v} \rightarrow \mathfrak{v}$ given by 
	$$\langle J_Z X, Y\rangle = \langle Z, [X, Y] \rangle,\,\,\,\,\,\,X,Y \in \mathfrak{v},\, Z\in \mathfrak{z},$$ satisfies the condition $J_Z^2=-\|Z\|^2I_{\mathfrak{v}}.$ Here, $I_{\mathfrak{v}}$ denotes the identity operator on $\mathfrak{v}.$ Kaplan \cite{Kaplan} proved that for $Z \in \mathfrak{z}$ with $\|Z\|=1$ one has $J_Z^2=-I_{\mathfrak{v}}$; that is, $J_Z$ induces a complex structure on $\mathfrak{v}.$ Therefore, $m=\dim(\mathfrak{v})$ is always even.
	
	A connected and simply connected Lie group $N$ is said to be $H$-type if its Lie algebra is an $H$-type algebra. Since $\mathfrak{n}$ is nilpotent, it follows that the exponential map is a global diffeomorphism. Hence, the elements of $N=\exp{\mathfrak{n}}$ can be parametrized by $(X, Z)$, for $X \in \mathfrak{v}$  and $Z \in \mathfrak{z}.$ From the Campbell-Baker-Hausdorff formula, the multiplication on $N$ is given by
	$$(X, Z) (X',Z')= \left(X+X', Z+Z'+\frac{1}{2}[X,X']\right).$$ 
 Note that the group $A= \mathbb{R}_+^*$ acts on $N$ by nonisotropic dilations: $(X, Y) \mapsto (a^{\frac{1}{2}}X, aZ).$ Therefore, by setting $\dim(\mathfrak{z})=k,$ the homogeneous dimension of $N$ is given by   $Q=\frac{m}{2}+k.$  At times, we also use symbol $\rho$ for $\frac{Q}{2}.$ Hence, $\dim(\mathfrak{s})=m+k+1,$ denoted by $d.$ 
	
	Let $S=N \ltimes A$ be the semidirect product of $N$ with $A$ under the aforementioned action. Therefore, the group multiplication on $S$ is defined by 
	$$(X,Z,a)(X',Z',a')= \left(X+a^{\frac{1}{2}}X', Z+aZ'+\frac{1}{2}a^{\frac{1}{2}} [X, X'], aa'\right).$$
	Then, $S$ is a solvable (connected and simply connected) Lie group with Lie algebra $\mathfrak{s}=\mathfrak{z}\oplus \mathfrak{v} \oplus \mathbb{R}$ and Lie bracket 
	$$[(X, Z, \ell), (X', Z', \ell')]= \left(\frac{1}{2} \ell X'-\frac{1}{2} \ell'X, \ell Z'-\ell'Z+[X, X]', 0\right).$$
	The group $S$ is equipped with the left-invariant Riemannian metric induced by 
	$$ \langle (X, Z, \ell), (X', Z', \ell')\rangle= \langle X, X'\rangle+\langle Z, Z'\rangle+\ell \ell'$$ on $\mathfrak{s}.$  The associated left Haar measure $dx$ on the group $S$ is given by $a^{-Q-1} dX\,dZ\,da=a^{-Q-1} dn\,da,$ where $dX,\, dZ$ and $da$ are the Lebesgue measures on $\mathfrak{v}, \mathfrak{z}$ and $\mathbb{R}_+^*,$ respectively. The elements of $A$ will be identified with $a_t=e^t,$ $t \in \mathbb{R}.$  We will also write any element $s \in S$ as $na_t$ by writing $S=NA.$ In particular, any element $a_t \in A$ can be thought as an element of $S$ by writing $a_t=e_N a_t,$ where $e_N$ is the identity element of $N.$ The group $S$ can be realized as the unit ball $B(\mathfrak{s})$ in $\mathfrak{s}$ using the Cayley transform $C: S \rightarrow  B(\mathfrak{s})$ (see \cite{Anker96}).
	
	To define the Helgason Fourier transform on the group $S,$ we need to describe the notion of the Poisson kernel (\cite{ACB97}). The Poisson Kernel $\mathcal{P}:S \times N \rightarrow \mathbb{R}$ is defined by $\mathcal{P}(na_t, n')= P_{a_t}(n'^{-1}n),$ where $$ P_{a_t}(n)= P_{a_t}(X, Z)=C a_t^Q \left( \left(a_t+\frac{|X|^2}{4} \right)^2+|Z|^2 \right)^{-Q},\,\,\,\, n=(X, Z) \in N\mbox{ and } t\in\mathbb{R}.$$	The value of $C$ is appropriately chosen so that $\int_N P_a(n) dn=1$ and $P_1(n) \leq 1.$ Now, we list some useful properties of the Poisson kernel; see \cite{KRS,RS,ACB97}. For $\lambda \in \mathbb{C},$ the complex power of the Poisson kernel is given by
	$$\mathcal{P}_\lambda(x, n)= \mathcal{P}(x, n)^{\frac{1}{2}-\frac{i \lambda}{Q}}.$$ It is known  (\cite{RS, ACB97}) that for each fixed $x \in S,$ $\mathcal{P}_\lambda(x, \cdot) \in L^p(N)$ for $1 \leq p \leq \infty$ if $\lambda = i \gamma_p \rho,$ where $\gamma_p= \frac{2}{p}-1.$ An important feature of the Poisson kernel $\mathcal{P}_\lambda(x,n)$ is that it is constant on the hypersurfaces $H_{n, a_t}=\{n \sigma(a_t n'): \, n' \in N\}.$ Here, $\sigma$ denotes the geodesic inversion on $S;$ see \cite{CDK91}.
	
	Let $\Delta_S$ denote the Laplace-Beltrami operator on $S.$ Then, for every fixed $n \in N,$ the function $\mathcal{P}_\lambda(x, n)$ is an eigenfunction of $\Delta_S$ with eigenvalue $-(\lambda^2+\frac{Q^2}{4});$ see \cite{ACB97}. The Helgason Fourier transform of a measurable function $f$ on $S$ is given by $$\widetilde{f}(\lambda, n)= \int_S f(x)\, \mathcal{P}_\lambda(x, n) dx$$ provided that the integral exists. 
	
	For $f \in C_c^\infty(S),$ the Fourier inversion formula \cite[Theorem 4.4]{ACB97} is given by
	$$f(x)= C \int_{\mathbb{R}} \int_N \widetilde{f}(\lambda, n)\,\mathcal{P}_{-\lambda}(x, n) |c(\lambda)|^{-2} \, d\lambda dn.$$ The Plancherel theorem \cite[Theorem 5.1]{ACB97} states that the Helgason Fourier transform extends to an isometry from $L^2(S)$ onto the space $L^2(\mathbb{R}_+ \times N, |c(\lambda)|^{-2}d\lambda dn).$ For precise value of the constants; see \cite{ACB97}. The function $|c(\lambda)|$ satisfies:
	\begin{align}\label{measureconstant}
		|c(\lambda)|^{-2}\asymp \begin{cases} \lambda^2\quad & \mbox{ if } \lambda \in [0,1] \\ \lambda^{d-1} \quad& \mbox{ if } \lambda>1\end{cases}.
	\end{align} 
	 See \cite[Theorem 1.14]{MV} and \cite[Lemma 4.8]{RS}.

	Let $e$ denote the identity element of $S$ and let $\tilde{\mu}$ be the metric induced by the canonical left invariant Riemannian structure on $S.$ A function $f$ on $S$ is said to be {\it radial} if for all $x, y \in S, $ $f(x)=f(y)$ whenever $\tilde{\mu}(x,e)=\tilde{\mu}(y,e).$ Note that the radial functions on $S$ can be identified with the functions $f=f(r)$ of the geodesic distance  $r=\tilde{\mu}(x, e) \in [0, \infty).$  Clearly,  $\tilde{\mu}(a_t, e)=|t|$ for $t \in \mathbb{R},$ where we have identified $a_t$ with $e_Na_t \in S:=NA$ with $e_N$ being the identity of $N.$ For any radial function $f,$ sometimes we use the notation $f(a_t)=f(t).$ The elementary spherical function $\phi_\lambda(x)$ is given by $$\phi_\lambda(x) :=\int_N \mathcal{P}_\lambda(x, n) \mathcal{P}_{-\lambda}(x, n)\, dn.  $$
	Note that the function $\phi_\lambda$ is a radial eigenfunction of the Laplace-Beltrami operator  $\Delta_S$ with eigenvalue $-(\lambda^2+\frac{Q^2}{4});$ see \cite{Anker96, ACB97}. Also, $\phi_\lambda(x)=\phi_{-\lambda}(x),\,\, \phi_\lambda(x)=\phi_\lambda(x^{-1})$ and $\phi_\lambda(e)=1.$ 
	In \cite{Anker96}, the authors proved that the radial part (in geodesic polar coordinates) of the Laplace-Beltrami operator $\Delta_S$ given by 
	$$\textnormal{rad}\, \Delta_S= \frac{\partial^2}{\partial t}+\left\{ \frac{m+l}{2} \coth{ \frac{t}{2}}+\frac{k}{2} \tanh{\frac{t}{2}} \right\} \frac{\partial}{\partial t},$$
	is (by substituting $r=\frac{t}{2}$) equal to  $\frac{1}{4} \mathcal{L}_{\alpha, \beta}$  with indices $\alpha=\frac{m+l+1}{2}$ and $\beta=\frac{l-1}{2}.$ Here, $\mathcal{L}_{\alpha, \beta}$ is the Jacobi operator and Koornwinder \cite{Koorn} treated it in detail. It is worth pointing out that we are in the ideal condition of the Jacobi analysis with $\alpha>\beta>\frac{-1}{2}.$ To be more precise, the Jacobi functions $\phi_\lambda^{\alpha, \beta}$ are related to the elementary spherical functions $\phi_\lambda$ by  $\phi_\lambda(t)=\phi_{2 \lambda}^{\alpha, \beta}(\frac{t}{2});$ see \cite{Anker96}. Hence, we have the following important lemma; see \cite{Platonov}.
	\begin{lem} \label{estijac}  Let $t, \lambda \in \mathbb{R}_+.$ Then,
		\begin{itemize}
			\item $|\phi_\lambda(t)| \leq 1.$
			\item $|1-\phi_\lambda(t)| \leq \frac{t^2}{2} \Big(4 \lambda^2+\frac{Q^2}{4}\Big).$
			\item There exists a constant $C>0$ such that $|1-\phi_\lambda(t)| \geq C$ for  $\lambda t \geq 1.$ 
		\end{itemize}
	\end{lem}

	Let $\sigma_t$ denote the normalized surface measure induced by the left invariant Riemannian metric on the geodesic sphere $S_t=\{y \in S: \tilde{\mu}(y, e)=t \}$ of radius $t$. Then, $\sigma_t$ is a nonnegative radial measure. The spherical mean operator $M_t$ is given by $M_tf: =f*\sigma_t$ for a suitable function $f$ on $S.$ Note that $M_tf(x)=\mathcal{R}(f^x)(t).$ Here, $f^x$ is the right translation of function $f$ by $x$ and $\mathcal{R}$ is the radialization operator given by
	$$\mathcal{R}f(x)=\int_{S_\nu} f(y)\,d\sigma_{\nu}(y)$$ for a suitable function $f$ and $\nu=r(x)= \tilde{\mu}(C(x), 0),$ where $C$ is the Cayley transform. It is easy to check that $\mathcal{R}f$ is a radial function. Also, $\mathcal{R}f=f$ for any radial function $f.$ Thus, for a radial function $f,$\, $M_tf$ is the usual translation of $f$ by $t.$ The spherical mean operator  $M_t$ is a bounded linear operator on $L^p(S)$ and  
	\begin{align}\label{FTTrans}
		\widetilde{M_t f}(\lambda, n)= \widetilde{f}(\lambda, n) \phi_\lambda(t)
	\end{align} for a suitable function $f$ on $S.$ Further, $M_tf$ converges to $f$ as $t \rightarrow 0.$ See \cite{KRS}.

	\subsection{Moduli of continuity of higher orders}
	
	Let $\omega$ be a mapping from $I\subset\mathbb{R}$ to the set $[0,\infty).$ The map $\omega$ is said to be almost increasing if there exists a constant $C\geq 1$ such that $\omega(t)\leq C\omega(s)$ whenever $t\leq s$ and $t,s\in I.$ The map $w$ is said to be almost decreasing if there exists a constant $C\geq 1$ such that $\omega(t)\leq C\omega(s)$ whenever $t\geq s$ and $t,s\in I.$
	
	Let $\delta_0>0$ and $k\in\mathbb{R}_+.$ A continuous function $\omega_k:[0,\delta_0]\rightarrow\mathbb{R}_+$ is said to be a {\it $k$th order  modulus of continuity} if $\omega_k(0)=0,$ $\omega_k(t)$ is almost increasing on $t\in[0,\delta_0]$ and $\frac{\omega_k(t)}{t^k}$ is almost decreasing on $t\in[0,\delta_0].$ Note that if $\omega$ is a $k$th-order modulus of continuity then $\omega$ is also an $m$th-order modulus of continuity for all $m\geq k.$
	
	We say that the $k$th-order modulus of continuity $\omega_k$ belongs to the Zygmund class $\mathscr{Z}^0$ if there exists a constant $C>0$ such that $$\int_0^t\frac{\omega_k(s)}{s}\, ds\leq C\omega_k(t),\ t\in [0,\delta_0].$$ We say that the $k$th-order modulus of continuity $\omega_k$ belongs to the Zygmund class $\mathscr{Z}_k$ if there exists a constant $C$ such that $$\int_t^{\delta_0}\frac{\omega_k(s)}{s^{1+k}}\, ds\leq C\frac{\omega_k(t)}{t^k},\ t\in [0,\delta_0].$$ The class  $\mathscr{Z}^0\cap\mathscr{Z}_k$ is called the  Zygmund-Bari-Stechkin class. Some classical examples in the Zygmund-Bari-Stechkin class are $t^\alpha,\ t^\alpha\left(\ln{\frac{1}{t}}\right)^\gamma$ and $t^\alpha\left(\ln\ln\frac{1}{t}\right)^\gamma,$ where $\alpha\in(0,k)$ and $\gamma\in\mathbb{R}.$ For more details on Zygmund classes, see \cite{KMRS}. 
	
	The crucial behaviour of the functions discussed above is near zero. In Theorem \ref{ConverseTitOP}, we also need to do certain estimations over $[\delta_0,\infty)$. Note that there is no loss of generality of the results in assuming certain restrictions, as in Theorem \ref{ConverseTitOP}, on $\omega_k$ over the interval $[\delta_0,\infty).$

	\section{Titchmarsh theorems on Damek-Ricci spaces}
	In this section, we present the main result of this paper that provides a description of generalised Lipschitz class functions in terms of the Helgason Fourier transform. We also prove the decay properties of the Helgason Fourier transform for the functions in the generalized Besov spaces. Lastly, we discuss certain integrability properties of the Helgason Fourier transform for the functions in the H\"older Lipschitz spaces.
	
	Now, we begin with the following definition of generalised Lipschitz class. For that, let us denote the harmonic $NA$ group by $S$ and the $k$th-order modulus of continuity by $\omega_k$.
	
	\begin{defn}
		A function $f\in L^2(S)$ is said to be in the generalised Lipschitz class $\mathrm{Lip}(\omega_k)$ if there exists a positive constant $C$ such that for all sufficiently small $t\in (0,1)$ we have $\|M_tf-f\|_2\leq C\omega_k(t).$
	\end{defn}
	
	The following theorem provides a growth property of image of the generalised Lipschitz functions under the Helgason Fourier transform.
	
	\begin{thm}\label{TitOP}
		Let $\omega_k$ be a $k$th-order modulus of continuity. If $f\in\mathrm{Lip}(\omega_k)$ then there exists a positive constant $C$ such that for all sufficiently small $t\in (0,1),$ $$\int_{\frac{1}{t}}^\infty\int_N|\widetilde{f}(\lambda,n)|^2\,dn\,d\lambda\leq Ct^{d-1}\omega_k(t)^2.$$
	\end{thm}
	
	\begin{proof}
		Using the Plancherel theorem and (\ref{FTTrans}) we have
		\begin{align*}
			\|M_tf-f\|_2^2=&\int_{\mathbb{R}_+}\int_N|\widetilde{(M_tf-f)}(\lambda,n)|^2|c(\lambda)|^{-2}\,dn\,d\lambda\\=&\int_{\mathbb{R}_+}\int_N|1-\phi_\lambda(a_t)|^2|\widetilde{f}(\lambda,n)|^2|c(\lambda)|^{-2}\,dn\,d\lambda.
		\end{align*}
		Since $f\in\mathrm{Lip}(\omega_k),$ it follows that
		\begin{align}\label{TransNorm}
			\int_{\mathbb{R}_+}\int_N|1-\phi_\lambda(a_t)|^2|\widetilde{f}(\lambda,n)|^2|c(\lambda)|^{-2}\,dn\,d\lambda\leq C^2\omega_k(t)^2.
		\end{align}
		
		Now,
		\begin{align*}
			\int_{\frac{1}{t}}^\infty\int_N|\widetilde{f}(\lambda,n)|^2\,dn\,d\lambda=&t^{d-1}\int_{\frac{1}{t}}^\infty \int_N|\widetilde{f}(\lambda,n)|^2\,dn\frac{1}{t^{d-1}}\,d\lambda\\\leq&t^{d-1}\int_{\frac{1}{t}}^\infty \int_N|\widetilde{f}(\lambda,n)|^2\,dn\lambda^{d-1}\,d\lambda.
		\end{align*}
		Since $\frac{1}{t}>1,$ then using (\ref{measureconstant}) we obtain  
		$$\int_{\frac{1}{t}}^\infty\int_N|\widetilde{f}(\lambda,n)|^2\,dn\,d\lambda\leq Ct^{d-1}\int_{\frac{1}{t}}^\infty \int_N|\widetilde{f}(\lambda,n)|^2\,dn|c(\lambda)|^{-2}\,d\lambda.$$
		
		Using Lemma \ref{estijac} we have 
		\begin{align*}
			\int_{\frac{1}{t}}^\infty\int_N|\widetilde{f}(\lambda,n)|^2\,dn\,d\lambda\leq& \frac{Ct^{d-1}}{C_1^2}\int_{\frac{1}{t}}^\infty |1-\phi_\lambda(a_t)|^2\int_N|\widetilde{f}(\lambda,n)|^2\,dn|c(\lambda)|^{-2}\,d\lambda\\\leq&\frac{Ct^{d-1}}{C_1^2}\int_0^\infty |1-\phi_\lambda(a_t)|^2\int_N|\widetilde{f}(\lambda,n)|^2\,dn|c(\lambda)|^{-2}\,d\lambda.
		\end{align*}
		Hence, by (\ref{TransNorm}) we have 
		
		\begin{align*}
			\int_{\frac{1}{t}}^\infty\int_N|\widetilde{f}(\lambda,n)|^2\,dn\,d\lambda\leq& C_2t^{d-1}\omega_k(t)^2.\qedhere
		\end{align*}
		\end{proof}
	
	We need the following lemma to prove a converse to the previous theorem.
	
	\begin{lem}\label{ConvLem}
	    	Let $\omega_k$ belongs to the Zygmund class $\mathscr{Z}^0$ and $f\in L^2(S).$ The following are equivalent. 
	    	\begin{enumerate}[(i)]
	    	    \item There exists a positive constant $C$ such that  $$\int_{\frac{1}{t}}^\infty\int_N|\widetilde{f}(\lambda,n)|^2\lambda^{d-1}\,dn\,d\lambda\leq C\omega_k(t)^2,\ t\in(0,1).$$
	    	    \item There exists a positive constant $C$ such that $$\int_{\frac{1}{t}}^{\frac{2}{t}}\int_N|\widetilde{f}(\lambda,n)|^2\lambda^{d-1}\,dn\,d\lambda\leq C\omega_k(t)^2,\ t\in(0,1).$$
	    	\end{enumerate}
	\end{lem}
	\begin{proof}
	    (i) implies (ii) is clear. Now, assume that (ii) holds.	Then, for $i\geq 0,$ we get $$\int_{\frac{2^i}{t}}^{\frac{2^{i+1}}{t}}\int_N|\widetilde{f}(\lambda,n)|^2\lambda^{d-1}\,dn\,d\lambda\leq C\omega_k\left(\frac{t}{2^i}\right)^2.$$
		Therefore,
		\begin{align*}
			\int_{\frac{1}{t}}^\infty\int_N|\widetilde{f}(\lambda,n)|^2\lambda^{d-1}\,dn\,d\lambda=&\sum_{i=0}^\infty\int_{\frac{2^i}{t}}^{\frac{2^{i+1}}{t}}\int_N|\widetilde{f}(\lambda,n)|^2\lambda^{d-1}\,dn\,d\lambda\\\leq&C\sum_{i=0}^\infty\omega_k\left(\frac{t}{2^i}\right)^2.
		\end{align*}
		Let $\mu>0$ be such that $\mu<m(\omega_k).$ Here, $m(\omega_k)$ is the lower MO index \cite[Pg. 31]{KMRS}. Since $\omega_k$ belongs to the Zygmund class $\mathscr{Z}^0,$ it follows by \cite[Theorem 2.10]{KMRS} that $\frac{\omega_k(t)}{t^\mu}$ is almost increasing.
		Since $\frac{t}{2^i}\leq t,\ i\geq 0,$ it follows by the definition of almost increasing that $$\frac{\omega_k(\frac{t}{2^i})}{(\frac{t}{2^i})^\mu}\leq C_1\frac{\omega_k(t)}{t^\mu}.$$
		Therefore, we have 
		\begin{align*}
			\int_{\frac{1}{t}}^\infty\int_N|\widetilde{f}(\lambda,n)|^2\lambda^{d-1}\,dn\,d\lambda\leq&C_2\sum_{i=0}^\infty\left(\left(\frac{t}{2^i}\right)^\mu\frac{\omega_k(t)}{t^\mu}\right)^2\\=&C_2\omega_k(t)^2\sum_{i=0}^\infty\left(\frac{1}{2^i}\right)^{2\mu}.
		\end{align*}
		Hence, we have 
		\begin{align*}
			\int_{\frac{1}{t}}^\infty\int_N|\widetilde{f}(\lambda,n)|^2\lambda^{d-1}\,dn\,d\lambda\leq&C_3\omega_k(t)^2.\qedhere
		\end{align*}
		\end{proof}
		
	The following theorem is a converse to Theorem \ref{TitOP} under more assumptions.
	
	\begin{thm}\label{ConverseTitOP}
		Let $k\leq 2$ and assume that $\omega_k$ belongs to the Zygmund classes $\mathscr{Z}^0$ and $\mathscr{Z}_k.$ Let $\omega_k(t)$ be bounded below by a positive number on interval $[\delta_0,\infty)$ and let $\frac{\omega_k(t)^2}{t^5}\in L^1([\delta_0,\infty)).$ For a function $f\in L^2(S),$ if there exists a positive constant $C$ such that for all sufficiently small $t\in (0,1),$ we have
		\begin{align}\label{FTcondi}
			\int_{\frac{1}{t}}^\infty\int_N|\widetilde{f}(\lambda,n)|^2\,dn\,d\lambda\leq Ct^{d-1}\omega_k(t)^2,
		\end{align} then $f\in\mathrm{Lip}(\omega_k).$   
	\end{thm}
	
	\begin{proof}
		Using (\ref{FTcondi}) we have
		$$\int_{\frac{1}{t}}^{\frac{2}{t}}\int_N|\widetilde{f}(\lambda,n)|^2\lambda^{d-1}\,dn\,d\lambda\leq\left(\frac{2}{t}\right)^{d-1}\int_{\frac{1}{t}}^\infty\int_N|\widetilde{f}(\lambda,n)|^2\,dn\,d\lambda\leq C_1\omega_k(t)^2.$$
		Then, by Lemma \ref{ConvLem} we have 
		\begin{align}\label{FTesti}
			\int_{\frac{1}{t}}^\infty\int_N|\widetilde{f}(\lambda,n)|^2\lambda^{d-1}\,dn\,d\lambda\leq&C_2\omega_k(t)^2.
		\end{align}
		
		As in the proof of Theorem \ref{TitOP}, using the Plancherel theorem and (\ref{FTTrans}) we get
		\begin{align}\label{norm=I1+I2}
			\|M_tf-f\|_2^2=&\int_{\mathbb{R}_+}\int_N|1-\phi_\lambda(a_t)|^2|\widetilde{f}(\lambda,n)|^2|c(\lambda)|^{-2}\,dn\,d\lambda=I_1+I_2.
		\end{align}
		Here, $$I_1=\int_0^{\frac{1}{t}}\int_N|1-\phi_\lambda(a_t)|^2|\widetilde{f}(\lambda,n)|^2|c(\lambda)|^{-2}\,dn\,d\lambda$$ and $$I_2=\int_{\frac{1}{t}}^\infty\int_N|1-\phi_\lambda(a_t)|^2|\widetilde{f}(\lambda,n)|^2|c(\lambda)|^{-2}\,dn\,d\lambda.$$
		
		First, we estimate $I_2.$ Since $\frac{1}{t}>1,$ using (\ref{measureconstant}) and Lemma \ref{estijac} we have
		$$I_2\leq 2^2C\int_{\frac{1}{t}}^\infty\int_N|\widetilde{f}(\lambda,n)|^2\lambda^{d-1}\,dn\,d\lambda.$$
		Thus, applying (\ref{FTesti}) we get
		\begin{align}\label{I2}
			I_2\leq 4C_3\omega_k(t)^2.
		\end{align}
		
		Now, we estimate $I_1.$ Using Lemma \ref{estijac} we have
		\begin{align*}
			I_1\leq&\int_0^{\frac{1}{t}}\int_N\left(\frac{t^2}{2}\left(4\lambda^2+\frac{Q^2}{4}\right)\right)^2|\widetilde{f}(\lambda,n)|^2|c(\lambda)|^{-2}\,dn\,d\lambda\\\leq&C_4t^4\int_0^{\frac{1}{t}}\int_N\lambda^4|\widetilde{f}(\lambda,n)|^2|c(\lambda)|^{-2}\,dn\,d\lambda\\&+C_5t^4\int_0^{\frac{1}{t}}\int_N|\widetilde{f}(\lambda,n)|^2|c(\lambda)|^{-2}\,dn\,d\lambda.
		\end{align*}
		By the Plancherel theorem in the second term, we get
		
		$$I_1\leq C_4t^4\int_0^{\frac{1}{t}}\int_N\lambda^4|\widetilde{f}(\lambda,n)|^2|c(\lambda)|^{-2}\,dn\,d\lambda+C_6t^4.$$
		
	Denote $\psi(s)=\int_s^\infty\int_N|\widetilde{f}(\lambda,n)|^2|c(\lambda)|^{-2}\,dn\,d\lambda,\ s\in\mathbb{R}_+.$ By the Plancherel theorem, $\psi(s)$ is bounded on $\mathbb{R}_+.$ Then, using the assumption that $\omega_k(s)$ is bounded below by a positive number on $[\delta_0,\infty),$ and (\ref{FTesti}), we have $\psi(s)\leq C_3\omega_k\left(\frac{1}{s}\right)^2,\ s\in\mathbb{R}_+.$ Therefore,
		
		\begin{align*}
			I_1\leq& C_4t^4\int_0^{\frac{1}{t}}\lambda^4(-\psi'(\lambda))\,d\lambda+C_6t^4\\=&C_4t^4\left(-\frac{1}{t^4}\psi\left(\frac{1}{t}\right)+\int_0^{\frac{1}{t}}4\lambda^3\psi(\lambda)\,d\lambda\right)+C_6t^4\\\leq&C_4t^4\int_0^{\frac{1}{t}}4\lambda^3\psi(\lambda)\,d\lambda+C_6t^4\\\leq&C_7t^4\int_0^{\frac{1}{t}}\lambda^3\omega_k\left(\frac{1}{\lambda}\right)^2\,d\lambda+C_6t^4\\=&C_7t^4\int_t^\infty\frac{\omega_k(\lambda)^2}{\lambda^5}\,d\lambda+C_6t^4.
		\end{align*}
		Using the assumption $\frac{\omega_k(t)^2}{t^5}\in L^1([\delta_0,\infty)),$ we get
		\begin{align*}
			I_1\leq&C_7t^4\int_t^{\delta_0}\frac{\omega_k(\lambda)^2}{\lambda^5}\,d\lambda+C_8t^4.
		\end{align*}
		Since $k\leq 2$ then $4-k\geq k.$ Therefore, $\frac{w_k(\lambda)}{\lambda^{4-k}}$ is almost decreasing. Then, we have 
		\begin{align*}
			I_1\leq&C_7t^4\frac{\omega_k(t)}{t^{4-k}}\int_t^{\delta_0}\frac{\omega_k(\lambda)}{\lambda^{k+1}}\,d\lambda+C_8t^4.
		\end{align*}
		Using the assumption that $\omega_k$ belongs to the Zygmund class $\mathscr{Z}_k$ in the first term and using the the fact that $k\leq 2$ implies that $\frac{w_k(t)}{t^2}$ is almost decreasing which implies that $t^2\leq C\omega_k(t)$ in the second term, it follows that
		\begin{align}\label{I1}
			I_1\leq&C_9\omega_k(t)^2.
		\end{align}
		
		Hence, the result follows form (\ref{norm=I1+I2}), (\ref{I1}) and (\ref{I2}).
	\end{proof}
	
	The following result is the main result of the paper and it is a direct consequence of Theorem \ref{TitOP} and Theorem \ref{ConverseTitOP}.
	
	\begin{cor}\label{TitThm}
		Let $k\leq 2$ and assume that $\omega_k$ belongs to the Zygmund classes $\mathscr{Z}^0$ and $\mathscr{Z}_k.$ Let $\omega_k(t)$ be bounded below by a positive number on interval $[\delta_0,\infty)$ and let $\frac{\omega_k(t)^2}{t^5}\in L^1([\delta_0,\infty)).$ A function $f\in L^2(S)$ belongs to the generalised Lipschitz class $\mathrm{Lip}(\omega_k)$ if and only if there exists a positive constant $C$ such that for all sufficiently small $t\in (0,1),$ we have
		\begin{align*}
			\int_{\frac{1}{t}}^\infty\int_N|\widetilde{f}(\lambda,n)|^2\,dn\,d\lambda\leq Ct^{d-1}\omega_k(t)^2.
		\end{align*}
	\end{cor}

	If we take $\omega_2(t)=t^\alpha\left(\ln{\frac{1}{t}}\right)^\gamma,$ where $\alpha\in (0,2), \gamma\in\mathbb{R}$ and $t\in[0,
	\delta_0],$ and $\omega_2$ is constant on $[\delta_0,\infty)$, then as a special case of Corollary \ref{TitThm} we obtain the following result.
	
	\begin{cor}\label{LipCor}
		Let $\alpha\in (0,2), \gamma\in\mathbb{R}$ and $f\in L^2(S).$ Then, we have $$\|M_tf-f\|_2\leq C_1t^\alpha\left(\ln{\frac{1}{t}}\right)^\gamma,$$ for all sufficiently small $t\in (0,1)$ if and only if 
		$$\int_{\frac{1}{t}}^\infty\int_N|\widetilde{f}(\lambda,n)|^2\,dn\,d\lambda\leq C_2t^{2\alpha+d-1}\left(\ln{\frac{1}{t}}\right)^{2\gamma},$$  for all sufficiently small $t\in (0,1).$
	\end{cor}
	
	The following result is an analogue of \cite[Theorem 1.5 (A)]{J} which gives decay properties of the Helgason Fourier transform for functions in the generalized Besov spaces.  
	
\begin{thm}
    Let $\alpha>0$ and $f\in L^2(S).$ If
    \begin{align}\label{TransC}
        \int_0^1\left(\frac{\|M_tf-f\|_2}{t^\alpha}\right)^2\,\frac{dt}{t}<\infty,
    \end{align} then 
    \begin{align}\label{FTC}
    \int_0^\infty\int_{\lambda=t}^{2t}\int_N\left(\frac{|\widetilde{f}(\lambda,n)|}{\lambda^{-\alpha}}\right)^2\,dn\,|c(\lambda)|^{-2}d\lambda\,\frac{dt}{t}<\infty.    
    \end{align}
    \end{thm}

\begin{proof}
    Let $f\in L^2(S)$ be such that (\ref{TransC}) holds. We decompose $$\int_0^\infty\int_{\lambda=t}^{2t}\int_N\left(\frac{|\widetilde{f}(\lambda,n)|}{\lambda^{-\alpha}}\right)^2\,dn\,|c(\lambda)|^{-2}d\lambda\,\frac{dt}{t}=I_1+I_2,$$ where $$I_1=\int_0^{\frac{1}{2}}\int_{\lambda=t}^{2t}\int_N\left(\frac{|\widetilde{f}(\lambda,n)|}{\lambda^{-\alpha}}\right)^2\,dn\,|c(\lambda)|^{-2}d\lambda\,\frac{dt}{t}$$ and $$I_2=\int_{\frac{1}{2}}^\infty\int_{\lambda=t}^{2t}\int_N\left(\frac{|\widetilde{f}(\lambda,n)|}{\lambda^{-\alpha}}\right)^2\,dn\,|c(\lambda)|^{-2}d\lambda\,\frac{dt}{t}.$$
    
    First, we estimate $I_1.$ Using the Plancherel theorem we have
    \begin{align*}
        I_1\leq&C_1\int_0^{\frac{1}{2}}t^{2\alpha}\int_{\lambda=t}^{2t}\int_N|\widetilde{f}(\lambda,n)|^2\,dn\,|c(\lambda)|^{-2}d\lambda\,\frac{dt}{t}\\\leq&C_2\|f\|_2^2\int_0^{\frac{1}{2}}t^{2\alpha}\,\frac{dt}{t}<\infty.
    \end{align*}
    
    Now, we estimate $I_2.$ By the change of variables $t=\frac{1}{2s}$ we obtain
    \begin{align*}
        I_2=&\int_0^1\int_{\lambda=\frac{1}{2s}}^{\frac{1}{s}}\int_N\left(\frac{|\widetilde{f}(\lambda,n)|}{\lambda^{-\alpha}}\right)^2\,dn\,|c(\lambda)|^{-2}d\lambda\,\frac{ds}{s}\\\leq&C_3\int_0^1\frac{1}{s^{2\alpha}}\int_{\lambda=\frac{1}{2s}}^{\frac{1}{s}}\int_N|\widetilde{f}(\lambda,n)|^2\,dn\,|c(\lambda)|^{-2}d\lambda\,\frac{ds}{s}.
    \end{align*}
    Since $s\lambda\leq 1,$ it follows by \cite[Theorem 4.7 (a)]{KRS} that
    \begin{align*}
        I_2\leq&C_4\int_0^1\frac{1}{s^{2\alpha}}\|M_sf-f\|_2^2\,\frac{ds}{s}.
    \end{align*}
    Therefore, by (\ref{TransC}), $I_2$ is finite. Hence (\ref{FTC}) holds.
\end{proof}

Now, we deal with the analogue of the \cite[Theorem 37]{Titch}. This provides an improvement of the well known Hausdorff-Young inequality. For $q=p',$ the next theorem is also proved in \cite{KR}.

Let $\alpha\in(0,1]$ and $p\geq 1.$ A function $f\in L^p(S)$ is said to be in the H\"older-Lipschitz space $Lip(\alpha;p)$ if there exists a positive constant $C$ such that for all sufficiently small $t\in (0,1)$ we have $\|M_tf-f\|_p\leq Ct^\alpha.$ 

\begin{thm} Let $\alpha\in(0,1],$ $p\in (1,2]$ and $q\in [p,p'].$ Assume that $f\in Lip(\alpha;p)$ and define $$F(\lambda):=\left(\int_N|\widetilde{f}(\lambda+i\gamma_q\rho,n)|^q\,dn\right)^{\frac{1}{q}}.$$ Then, $F\in L^\beta(\mathbb{R}_+,|c(\lambda)|^{-2}\,d\lambda),$ for any $\beta\in\left(\frac{dp'}{d+\alpha p'},p'\right].$
\end{thm}
\begin{proof}
	 Using \cite[Theorem 4.7]{KRS} we get
	 $$\int_{\mathbb{R}} \min\{1,(\lambda t)^{2p'}\} F(\lambda)^{p'} |c(\lambda)|^{-2}\,d\lambda  \leq C \|M_tf-f\|_{p}^{p'}.$$
	 Then, for small enough $t$ we have 
	\begin{equation}\label{Est1}
	\int_{|\lambda|<\frac{1}{t}} \lambda^{2p'} F(\lambda)^{p'} |c(\lambda)|^{-2}\,d\lambda  \leq C \left( \frac{\|M_tf-f\|_{p}}{t^2} \right)^{p'}\leq C_1t^{(\alpha-2)p'}.
	\end{equation}
	
	For $s=\frac{1}{t}$ and $\beta< p',$ define $$\phi(s):=\int_1^s \lambda^{2\beta} F(\lambda)^\beta |c(\lambda)|^{-2}\,d\lambda.$$
	By the H\"older inequality with $\frac{\beta}{p'}+\left(1-\frac{\beta}{p'}\right)=1$ and using (\ref{measureconstant}), we get
	\begin{align*}
		\phi(s) & \leq 
		C_2 \left( \int_{1}^s \lambda^{2p'} F(\lambda)^{p'} |c(\lambda)|^{-2}\, d\lambda \right)^{\frac{\beta}{p'}} \left( \int_{1}^s |c(\lambda)|^{-2}\, d\lambda \right)^{1-\frac{\beta}{p'}} 
		\\& \leq C_2 \left( \int_{1}^s \lambda^{2p'} F(\lambda)^{p'} |c(\lambda)|^{-2} d\lambda \right)^{\frac{\beta}{p'}} \left( \int_{1}^s \lambda^{d-1}\, d\lambda \right)^{1-\frac{\beta}{p'}}.
	\end{align*}
	Using (\ref{Est1}) we have
	$$\phi(s)\leq C_2 s^{(2-\alpha) \beta + d\left(1-\frac{\beta}{p'}\right)}.$$
	
	Hence, using the integration by parts, we obtain 
	\begin{align*}
		\int_{1}^s F(\lambda)^{\beta} |c(\lambda)|^{-2}\, d\lambda &= \int_{1}^{s} \lambda^{-2\beta}\, \phi'(\lambda)\, d\lambda \\&= s^{-2\beta} \phi(s)+ 2 \beta \int_1^{s} \lambda^{-2\beta-1} \phi(\lambda)\, d\lambda \\&\leq C_2 s^{-\alpha \beta + d\left(1-\frac{\beta}{p'}\right)} + C_3\int_1^s \lambda^{-1-\alpha \beta + d\left(1-\frac{\beta}{p'}\right)}\,d\lambda\\&=O(1)
	\end{align*}
	provided that $-\alpha \beta + d\left(1-\frac{\beta}{p'}\right)<0,$ that is, $\beta>\frac{dp'}{d+\alpha p'}.$
\end{proof}
	
	\section*{Acknowledgment}
	Manoj Kumar is supported by the NBHM post-doctoral fellowship with Ref. number: 0204/3/2021/R\&D-II/7356. Vishvesh Kumar and Michael Ruzhansky are supported  by the FWO Odysseus 1 grant G.0H94.18N: Analysis and Partial
	Differential Equations, the Methusalem programme of the Ghent University Special Research Fund (BOF) (Grant number 01M01021) and by FWO Senior Research Grant G011522N. MR is also supported by EPSRC grant EP/R003025/2.

	\bibliographystyle{amsplain}

\begin{thebibliography}{99}
		
		
		\bibitem{Anker96}  Anker, J-P., Damek, E., Yacoub, C. Spherical analysis on harmonic AN groups. {\it Ann. Scuola Norm. Sup. Pisa Cl. Sci.} (4) 23 (1996), no. 4, 643-679. 
		
		\bibitem{Astengo} Astengo, F. A class of $L^p$ convolutors on harmonic extensions of H-type groups. {\it J. Lie Theory} 5 (1995), no. 2, 147-164.
		\bibitem{Astengo2}  Astengo, F. Multipliers for a distinguished Laplacean on solvable extensions of $H$-type groups. {\it Monatsh. Math.} 120 (1995), no. 3-4, 179-188.
		
		\bibitem{ACB97}  Astengo, F., Camporesi, R., Di Blasio, B. The Helgason Fourier transform on a class of nonsymmetric harmonic spaces. {\it Bull. Austral. Math. Soc.} 55 (1997), no. 3, 405-424.
		
		
		\bibitem{Bray}  Bray, W. O. Growth and integrability of Fourier transforms on Euclidean space. {\it J. Fourier Anal. Appl.} 20 (2014), no. 6, 1234-1256.
		\bibitem{BrayPinsky}  Bray, W. O., Pinsky, M. A. Growth properties of Fourier transforms via moduli of continuity. {\it J. Funct. Anal.} 255 (2008), no. 9, 2265-2285.
		
		\bibitem{CDK91} Cowling, M., Dooley A., Koranyi, A., Ricci, F. $H$-type groups and Iwasawa decompositions, {\it Adv. Math.} 87 (1991), no. 1, 1-41.
		
		
		\bibitem{Damek} Damek, E. The geometry of a semidirect extension of a Heisenberg type nilpotent group. {\it Colloq. Math.} 53 (1987), no. 2, 255-268.
		\bibitem{DamekRicci} Damek, E., Ricci, F. A class of nonsymmetric harmonic Riemannian spaces. {\it Bull. Amer. Math. Soc.} 27 (1992), no. 1, 139–142.
		\bibitem{DamekRicci1} Damek, E., Ricci, F. Harmonic analysis on solvable extensions of H-type groups. {\it J. Geom. Anal.} 2 (1992), no. 3, 213–248.
		\bibitem{Di-Blasio} Di Blasio, B. Paley-Wiener type theorems on harmonic extensions of H-type groups. {\it Monatsh. Math.} 123 (1997), no. 1, 21–42.
		
		\bibitem{DDR} Daher, R., Delgado, J., Ruzhansky, M. Titchmarsh theorems for Fourier transforms of H\"older-Lipschitz functions on compact homogeneous manifolds. {\it Monatsh. Math.} 189 (2019), no. 1, 23-49. 
		
		\bibitem{DFR} Daher, R.; Fernandez, A.; Restrepo, J. E.  Characterising extended Lipschitz type conditions with moduli of continuity. {\it Results Math.} 76 (2021), no. 3, Paper No. 125, 18 pp.
		
		\bibitem{ED} El Ouadih, S.; Daher, R.
		Lipschitz conditions in Damek-Ricci spaces. {\it C. R. Math. Acad. Sci. Paris} 359 (2021), 675–685.
		
		
		\bibitem{FBK}  Fahlaoui, S., Boujeddaine, M., El Kassimi, M. Fourier transforms of Dini-Lipschitz functions on rank $1$ symmetric spaces. {\it Mediterr. J. Math.} 13 (2016), no. 6, 4401-4411.
		
		\bibitem{FRS} Fernandez, A.; Restrepo, J. E.; Suragan, D. Lipschitz and Fourier type conditions with moduli of continuity in rank $1$ symmetric spaces. {\it Monatsh. Math.} (2021),  https://doi.org/10.1007/s00605-021-01621-w.
		
		\bibitem{GT} Gorbachev, D., Tikhonov, S. Moduli of smoothness and growth properties of Fourier transforms: two-sided estimates. {\it J. Approx. Theory} 164(9), 1283–1312 (2012).
		\bibitem{J} Jord\~ao, T.: Decay of Fourier transforms and generalized Besov spaces. {\it Constr. Math. Anal.} 3(1), (2020), 20-35.
		\bibitem{Kaplan}  Kaplan, A. Fundamental solutions for a class of hypoelliptic PDE generated by composition of quadratic forms. {\it Trans. Amer. Math. Soc.} 258 (1980), no. 1, 147-153.
		\bibitem{KMRS} Kokilashvili, V., Meskhi, A., Rafeiro, H., Samko, S. {\it Integral operators in non-standard function spaces}. Volume 1: Variable exponent Lebesgue and amalgam spaces. Basel: Birkhäuser/Springer (2016).
		\bibitem{KRS}  Kumar, P., Ray, S. K., Sarkar, R. P. The role of restriction theorems in harmonic analysis on harmonic NA groups. {\it J. Funct. Anal.} 258 (2010), no. 7, 2453-2482. 
		\bibitem{KR1} Kumar, V.; Ruzhansky, M. A note on K-functional, modulus of smoothness, Jackson theorem and Bernstein-Nikolskii-Stechkin inequality on Damek-Ricci spaces. {\it J. Approx. Theory} 264 (2021), Paper No. 105537, 13 pp.
		\bibitem{KR} Kumar, V.; Ruzhansky, M. Titchmarsh theorems, Hausdorff-Young-Paley inequality and $L^p$-$L^q$ boundedness of Fourier multipliers on harmonic $NA$ groups, arXiv:2107.13044. 
		\bibitem{Koorn}  Koornwinder, Tom H. Jacobi functions and analysis on noncompact semisimple Lie groups. Special functions: group theoretical aspects and applications, 1-85, {\it Math. Appl.}, Reidel, Dordrecht, (1984).
	
		
		\bibitem{Platonov} Platonov, S.S. The Fourier transform of functions satisfying a Lipschitz condition on symmetric spaces of rank 1. {\it Sibirsk. Mat. Zh.} 46(6), 1374-1387 (2005).
		
		\bibitem{Platonov17} Platonov, S.S. An analogue of the Titchmarsh theorem for the Fourier transform on the group of p-adic numbers. {\it p-Adic Numbers Ultrametric Anal. Appl.} 9(2), (2017), 158-164.
		\bibitem{RS}  Ray, S. K., Sarkar, R. P. Fourier and Radon transform on harmonic NA groups. {\it Trans. Amer. Math. Soc.} 361 (2009), no. 8, 4269-4297.
		
		\bibitem{Titch} Titchmarsh E. C. {\it Introduction to the Theory of Fourier Integral}, 2nd edition, Oxford University Press, Oxford (1948).
		
		\bibitem{MV} Vallarino, M. {\it Analysis on harmonic extensions	of $H$-type groups}, Ph.D. Thesis, Università degli Studi di Milano–Bicocca, 2005. arXiv:1504.00329.
		\bibitem{Volo} Volosivets, S. S. Fourier transforms and generalized Lipschitz classes in uniform metric. {\it J. Math. Anal. Appl.} 383(2), 344–352 (2011).
		
		\bibitem{Younis}  Younis, M. S. Fourier transforms of Dini-Lipschitz functions. {\it Internat. J. Math. Math. Sci.} 9 (1986), no. 2, 301-312. 
		\bibitem{Younis2}  Younis, M. S. Fourier transforms of Dini-Lipschitz functions on Vilenkin groups. {\it Internat. J. Math. Math. Sci.} 15 (1992), no. 3, 609-612.
		\bibitem{Younis1}  Younis, M. S. Fourier transforms of Lipschitz functions on certain Lie groups. {\it Int. J. Math. Math. Sci.} 27 (2001), no. 7, 439-448.
	\end{thebibliography}

\end{document}